\documentclass[a4paper, 11pt]{paper}
\usepackage{amssymb}
\usepackage{amsmath}
\usepackage{amsthm}
\usepackage[mathscr]{euscript}
\usepackage{enumerate}
\usepackage[all]{xy}
\usepackage{graphics}
\usepackage[dvips]{graphicx}

\CompileMatrices
\newif\ifbqpd 
\bqpdtrue
\newif\ifhozna \hoznafalse

\begin{document}



\let\goth\mathfrak


\def\theoremname{Theorem}
\def\lemmaname{Lemma}
\def\corollaryname{Corollary}
\def\propositionname{Proposition}
\def\factname{Fact}
\def\remarkname{Remark}
\def\examplename{Example}
\def\notename{Note}

\newtheorem{thm}{\theoremname}[section]
\newtheorem{lem}[thm]{\lemmaname}
\newtheorem{cor}[thm]{\corollaryname}
\newtheorem{prop}[thm]{\propositionname}
\newtheorem{fact}[thm]{\factname}
\newtheorem{note}[thm]{\notename}
\newtheorem{exmx}[thm]{\examplename}
\newenvironment{exm}{\begin{exmx}\normalfont}{\end{exmx}}

\newtheorem{rem}[thm]{\remarkname}

\newtheorem*{note*}{\notename}

\def\myend{{}\hfill{\small$\bigcirc$}}

\newtheorem{cnstrx}[thm]{Construction}
\newenvironment{constr}{\begin{cnstrx}\normalfont}{\myend\end{cnstrx}}

\def\myend{{}\hfill{\small$\bigcirc$}}

\newenvironment{ctext}{%
  \par
  \smallskip
  \centering
}{%
 \par
 \smallskip
 \csname @endpetrue\endcsname
}

\newenvironment{cmath}{%
  \par
  \smallskip
  \centering
  $
}{%
  $
  \par
  \smallskip
  \csname @endpetrue\endcsname
}


\newcounter{sentence}
\def\thesentence{\roman{sentence}}
\def\labelsentence{\upshape(\thesentence)}

\newenvironment{sentences}{%
        \list{\labelsentence}
          {\usecounter{sentence}\def\makelabel##1{\hss\llap{##1}}
            \topsep3pt\leftmargin0pt\itemindent40pt\labelsep8pt}%
  }{%
    \endlist}


\newcounter{typek}\setcounter{typek}{0}
\def\thetypek{\roman{typek}}
\def\typitem#1{\refstepcounter{typek}\item[\normalfont(\roman{typek};\hspace{1ex}{#1})\quad]}
\newenvironment{typcription}{\begin{description}\itemsep-2pt\setcounter{typek}{0}}{%
\end{description}}
\newcounter{ostitem}\setcounter{ostitem}{0}


\newcommand*{\sub}{\raise.5ex\hbox{\ensuremath{\wp}}}
\def\Aut{\mathrm{Aut}}
\def\Fix{\mathrm{Fix}}
\newcommand*{\struct}[1]{{\ensuremath{\langle #1 \rangle}}}
\def\dom{\mathrm{Dom}}
\def\rng{\mathrm{Rng}}

\newcommand{\msub}{\mbox{\large$\goth y$}}    
\def\skros(#1,#2){{\{ #1,#2 \}}}
\def\LineOn(#1,#2){\overline{{#1},{#2}\rule{0em}{1,5ex}}}
\def\inc{\mathrel{\strut\rule{3pt}{0pt}\rule{1pt}{9pt}\rule{3pt}{0pt}\strut}}
\def\lines{{\cal L}}

\def\VerSpace(#1,#2){{\bf V}_{{#2}}({#1})}
\def\GrasSpace(#1,#2){{\bf G}_{{#2}}({#1})}
\def\GrasSpacex(#1,#2){{\bf G}^\ast_{{#2}}({#1})}
\def\VeblSpSymb{{\bf P}\mkern-10mu{\bf B}}
\def\VeblSpace(#1){\VeblSpSymb({#1})}

%
\def\konftyp(#1,#2,#3,#4){\left( {#1}_{#2}\, {#3}_{#4} \right)}
\newcount\liczbaa
\newcount\liczbab
\def\binkonf(#1,#2){\liczbaa=#2 \liczbab=#2 \advance\liczbab by -2
\def\doa{\ifnum\liczbaa = 0\relax \else
\ifnum\liczbaa < 0 \the\liczbaa \else +\the\liczbaa\fi\fi}
\def\dob{\ifnum\liczbab = 0\relax \else
\ifnum\liczbab < 0 \the\liczbab \else +\the\liczbab\fi\fi}
\konftyp(\binom{#1\doa}{2},#1\dob,\binom{#1\doa}{3},3) }
\let\binconf\binkonf

\def\perspace(#1,#2){\mbox{\boldmath$\Pi$}({#1},{#2})}
\def\starof(#1){{\mathop{\mathrm S}(#1)}}
\def\topof(#1){{\mathop{\mathrm T}(#1)}}
\def\bijfam{{\cal S}}
\def\id{\mathrm{id}}


\def\psts{partial Steiner triple system}
\def\PSTS{{\sf PSTS}}
\def\BSTS{{\sf BSTS}}
\def\STP{{\sf STP}}

\def\brak{\ifhozna\relax\else{\warning[TO BE COMPLETED]}\fi}



\title[Combinatorial quasi Veronesians]{%
Configurations related to combinatorial Veronesians representing a skew perspective}
\author{Agata Bazylewska-Zejer, Ma{\l}gorzata Pra{\.z}mowska, \\ Krzysztof Pra{\.z}mowski}

\maketitle

\begin{abstract}
  A combinatorial object representing schemas of, possibly skew, perspectives,
  called {\em a configuration of skew perspective} has been defined in \cite{klik:binom}, \cite{maszko}.
  Here we develop the theory of configurations generalizing perspectives defined in
  combinatorial Veronesians.
  The complete classification of thus obtained $\konftyp(15,4,20,3)$-configurations is presented.
\newline
{\em key words}: Veblen (Pasch) configuration, combinatorial Veronesian, 
binomial configuration, complete (free sub)subgraph, perspective.
\end{abstract}

\section*{Introduction}

A project to characterize and classify so called binomial \psts s via the arrangement of their
free complete subgraphs was started in \cite{klik:binom}.
In particular, we know that if a configuration $\goth K$ contains the maximal number (with respect to its parameters,
i.e. $= m+2$, where $m$ is the rank of a point in $\goth K$) of free $K_{m+1}$-subgraphs then $\goth K$
is a so called combinatorial Grassmannian (cf. \cite{perspect}) and if
$\goth K$ contains $m$ free complete subgraphs then it is a multi veblen configuration (cf. \cite{mveb2proj}).
One of the most fruitful observation used to obtain a required classification is quoted 
in \ref{fct:zepers} after \cite{klik:binom}:
\begin{quotation}\em
  a configuration $\goth K$ with two free subgraphs $K_{m+1}$ can be considered as a schema 
  of an abstract perspective between  these graphs.
\end{quotation}
Let us stress on the words {\em schema} and {\em abstract}: 
`ordinary' projections, as used and investigated e.g. in \cite{projmono2}, \cite{hilbert}, or \cite{projchain}
can be considered as examples (realizations) of our perspectives, but configurations considered in this
paper do not necessarily have any realization in a desarguesian projective space.

The above observation enables us to reduce the problem to a classification of `line perspectives'
(maps between edges of graphs, we call them also `skews') and a classification of `axial configurations'
(defined on intersection points of lines containing perspective edges); these axial configurations have
vertices with on $2$ smaller point rank.
If $\goth K$ has three free $K_{m+1}$, a similar technique involving a triple perspective can be used;
for $m=4$ the complete classification was given in \cite{STP3K5}.
If the line perspective preserves intersection of edges a simple theory presenting the case can be developed
(see \cite{maszko}). In result, the complete classification of such `cousins' of the Desargues configuration
for $m=4$ could be obtained -- and presented in \cite{maszko:male}.
Even in this small case $m=4$ there are, generally, ${10}! \geq 3\cdot 10^6$ admissible perspectives.
One has to look for some ways to distinguish among them some more regular and interesting.

On a second side, there is a family of known and investigated configurations other than combinatorial Grassmannians: 
combinatorial Veronesians.
This family contains binomial {\psts s} with exactly three maximal free subgraphs. Computing formulas which 
define in this case the line perspectives we obtain a class of functions that can determine an (interesting)
family of configurations: keeping invariant skew taken from the theory of combinatorial Veronesians 
we vary axial configurations.

Rudiments of the theory of so obtained Veronese-like perspectivities are presented in this note.
We close the paper with the complete classification of $\konftyp(15,4,20,3)$-configurations which can be 
presented as a such Veronese-like perspective: there are $18$ such configurations, and $14$ of them have
not been found before.

\tableofcontents

\section{Underlying ideas and basic definitions}\label{sec:under}

Let us begin with introducing some, standard, notation.
Let $X$ be an arbitrary set.
The symbol $\bijfam_X$ stands for the family of permutations of $X$.
Let $k$ be a positive integer; 
we write $\sub_k(X)$ for the family of $k$-element subsets of $X$.
Then $K_X = \struct{X,\sub_2(X)}$ is the complete graph on $X$;
$K_n$ is $K_X$ for any $X$ with $|X| = n$. Analogously, $\bijfam_n = \bijfam_X$.
\par\noindent
A {\em $\konftyp(\nu,r,b,\varkappa)$-configuration} is a configuration 
(a {\em partial linear space} i.e. an incidence structure with blocks ({\em lines})
pair wise intersecting in at most a point)
with
$\nu$ points, each of rank $r$, and $b$ lines, each of rank (size) $\varkappa$.
A {\em \psts}\ (in short: a \PSTS) is a partial linear space 
with all the lines of size $3$.
A $\binkonf(n,0)$-configuration is a \psts, it is called 
a {\em binomial \psts}.
\par\noindent
We say that a graph $\cal G$ is {\em freely contained} in a 
configuration $\goth B$ iff 
the vertices of $\cal G$ are points of $\goth B$, 
each edge $e$ of $\cal G$ is contained in a line $\overline{e}$ of $\goth B$,
the above map $e \mapsto \overline{e}$ is an injection, 
and lines of $\goth B$ which contain disjoint edges of $\cal G$ do not 
intersect in $\goth B$.
If $\goth B$ is a $\binkonf(n,0)$-configuration and ${\cal G} = K_X$
then $|X| + 1 \leq n$. Consequently, $K_{n-1}$ is a maximal complete graph
freely contained in a binomial $\binkonf(n,0)$-configuration.
Further details of this theory are presented in \cite{klik:binom},
relevant results will be quoted in the text, when needed.

\begin{constr}[{\cite[Constr. 1.1]{maszko}}]\label{def:pers}
  Let $I$ be a nonempty finite set, $n := |I|\geq 2$.
  In most parts, without loss of generality, we assume that
  $I = I_n =  \{ 1,\ldots,n \}$.
  Let 
    $A = \{a_i\colon i\in I \}$ and $B = \{b_i\colon i\in I\}$
  be two disjoint $n$-element sets, let $p\notin A\cup B$.
\newline
  Then we take a $\binom{n}{2}$-element set
  $C = \{c_u\colon u\in\sub_2(I)\}$
  disjoint with $A\cup B\cup\{p\}$.
  Set 
\begin{equation*}
  {\cal P} = A\cup B\cup\{p\}\cup C.
\end{equation*}
  Let us fix a permutation $\sigma$ of $\sub_2(I)$ and write
\begin{eqnarray*}
  {\cal L}_p & := & \big\{ \{ p,a_i,b_i \}\colon i\in I \big\},
  \\
  {\cal L}_A & := & \big\{ \{ a_i,a_j,c_{\{ i,j \}} \}\colon \{i,j\}\in\sub_2(I) \big\},
  \\
  {\cal L}_B & := & \big\{ \{ b_i,b_j,c_{\sigma^{-1}(\{ i,j \})} \} \colon \{i,j\}\in\sub_2(I) \big\}.
\end{eqnarray*}
  Finally, let ${\cal L}_C$ be a family of $3$-subsets of $C$ such that
  ${\goth N} = \struct{C,{\cal L}_C}$ is a 
  $\binconf(n,0)$-configuration.
  Set
$$
  {\cal L} = {\cal L}_p \cup {\cal L}_A\cup {\cal L}_B\cup {\cal L}_C
  \text{ and }
  \perspace(n,\sigma,{\goth N}) := \struct{{\cal P},{\cal L}}.
$$
  The structure $\perspace(n,\sigma,{\goth N})$ will be referred to as a 
  {\em skew perspective} with the {\em skew} $\sigma$.
\end{constr}
We frequently shorten $c_{\{i,j\}}$ to $c_{i,j}$.
Sometimes the parameter $\goth N$ will not be essential and then it will be omitted,
we shall write simply $\perspace(n,\sigma)$.
In essence, the names ``$a_i$", ``$c_{i,j}$" are -- from the point of view of 
mathematics -- arbitrary, and could be replaced by any other labelling
(cf. analogous problem of labelling in \cite[Constr. 3, Repr. 3]{pascvebl}
or in \cite[Rem 2.11, Rem 2,13]{STP3K5}, \cite[Exmpl. 2]{pascvebl}).
Formally, one can define $J = I \cup \{a,b\}$,
$x_i = \{x,i\}$ for $x\in\{a,b\} =: p$ and $i\in I$, and $c_u = u$ for $u\in \sub_2(I)$.
After this identification $\perspace(n,\sigma)$ becomes a structure defined
on $\sub_2(J)$.
Then, it is easily seen that
\begin{equation}\label{paramy}
  \perspace(n,\sigma,{\goth N}) \text{ is a }\binconf(n,+2)\text{ configuration}.
\end{equation}
In particular, it is a \psts,
so we can use standard notation:
$\LineOn(x,y)$ stands for the line which joins two collinear points $x,y\in{\cal P}$,
and then we define on $\cal P$ the partial operation $\oplus$ with the following requirements:
$x\oplus x = x$, $\{x,y,x\oplus y\}\in{\cal L}$ whenever $\LineOn(x,y)$ exists.
Observe that 
(cf. \cite[Eq. (1), the definition of 
{\em combinatorial Grassmannian} $\GrasSpace(n,2)$]{perspect})
\begin{equation}\label{gras:pers}
  \GrasSpace(n+2,2) = \GrasSpace(J,2) = \struct{\sub_2(J),\sub_3(J),\subset}
  \cong \perspace(n,\id_{I_n},{\GrasSpace(I_n,2)}).
\end{equation}

It is clear that $A^\ast = A \cup \{p\}$ and $B^\ast = B\cup\{p\}$ are two $K_{n+1}$-graphs
freely contained in $\perspace(n,\sigma,{\goth N})$.
Applying the results \cite[Prop. 2.6 and Thm. 2.12]{klik:binom} 
we immediately obtain the following fact.
\begin{fact}\label{fct:zepers}
  Let $N = n+2$. The following conditions are equivalent.
  \begin{sentences}\itemsep-2pt
  \item
    $\goth M$ is a binomial $\binkonf(N,0)$-configuration which freely contains 
    two $K_{N-1}$-graphs.
  \item
    ${\goth M}\cong \perspace(n,\sigma,{\goth N})$ for a $\sigma\in \bijfam_{\sub_2(I_n)}$
    and a $\binkonf(n,0)$-configuration $\goth N$ defined on $\sub_2(I_n)$.
  \end{sentences}
\end{fact}

The map
$$
  \pi = \big( a_i \longmapsto b_i,\; i\in I \big)
$$
is a {\em point-perspective} of $K_A$ onto $K_B$ with centre $p$.
Moreover, the map
$$
  \xi = \big( \LineOn(a_i,a_j) \longmapsto \LineOn(b_{i'},b_{j'}), \;
  \sigma( \{ i,j \} ) = \{i',j'\} \in \sub_2(I) \big)
$$
is a {\em line-perspective}, where $\goth N$ is the axial configuration of our perspective.
With each permutation $\sigma_0\in \bijfam_I$ we associate the permutation $\overline{\sigma_0}$ defined by
\begin{equation}\label{lem:meetinvar}    
          \overline{\sigma_0}(\{ i,j \})
	   =\{ \sigma_0(i),\sigma_0(j) \}
      \end{equation}
for every $\{i,j\}\in\sub_2(I)$.
\begin{note}\normalfont
  If $\sigma_0\in \bijfam_I$ we frequently identify $\sigma_0$, $\overline{\sigma_0}$,
  and the corresponding map $\xi$.
  Consequently, if $\sigma\in \bijfam_I$ we write
  $\perspace(n,\sigma,{\goth N})$ in place of
  $\perspace(n,\overline{\sigma},{\goth N})$.
\end{note}

\begin{prop}[{comp. \cite[Prop. 2.2]{maszko}}]\label{prop:iso0}
  Let $f \in \bijfam_{{\cal P}}$, $f(p) = p$, $\sigma_1,\sigma_2 \in \bijfam_{\sub_2(I)}$,
  and 
  ${\goth N}_1, {\goth N}_2$ be two $\binkonf(n,0)$- configurations defined
  on $\sub_2(I)$.
  The following conditions are equivalent.
  \begin{sentences}\itemsep-2pt
   \item\label{propiso0:war1}
     $f$ is an isomorphism 
     of $\perspace(n,\sigma_1,{\goth N}_1)$ onto $\perspace(n,\sigma_2,{\goth N}_2)$.
   \item\label{propiso0:war2}
     There is $\varphi \in \bijfam_I$ such that
     one of the following holds
     \begin{eqnarray} 
     \label{iso0:war1-1}
        \overline{\varphi} \text{ ( comp. \eqref{lem:meetinvar})} & {\text is } & \text{ an isomorphism of } 
            {\goth N}_1 \text{ onto } {\goth N}_2,
     \\ 
     \label{propiso0:typ1}
       f(x_i) = x_{\varphi(i)},\; x = a,b,\; &&
       f(c_{\{i,j\}}) = c_{\{\varphi(i),\varphi(j)\}},
       \quad i,j\in I, i\neq j,
     \\ \label{iso0:war2}
       \overline{\varphi} \circ \sigma_1 & = & \sigma_2 \circ \overline{\varphi},
     \end{eqnarray}
     or
     \begin{eqnarray} 
     \label{iso0:war1-2}
        \sigma_2^{-1}\overline{\varphi} & \text{  is } & \text{ an isomorphism of } 
            {\goth N}_1 \text{ onto } {\goth N}_2,
     \\ 
     \label{propiso0:typ2}
       f(a_i) = b_{\varphi(i)},\; f(b_i) = a_{\varphi(i)}, &&
        f(c_{\{i,j\}}) = c_{\sigma_2^{-1}\{\varphi(i),\varphi(j)\}},
       \; i,j\in I, i\neq j,
     \\ \label{iso0:war3}
       \overline{\varphi} \circ \sigma_1 & = & \sigma_2^{-1} \circ \overline{\varphi}.
     \end{eqnarray}
  \end{sentences}
\end{prop}

\begin{lem}[{Comp. \cite[Lem. 2.1]{maszko}}]\label{lem:nextgraf00}
    Assume that $\perspace(n,\sigma,{\goth N})$ freely contains a complete $K_{n+1}$-graph $G\neq K_{A^\ast},K_{B^\ast}$,
    $\sigma\in \bijfam_{\sub_2(I)}$.
    Then there is $i_0 \in I$ such that 
    $\starof(i_0) = \{ c_u\colon i_0\in u\in\sub_2(I) \}$
    is a collinearity clique in $\goth N$ freely contained in it
    and $\sigma$ satisfies
    \begin{equation}\label{war:nextgraf00}
      i_0 \in u \implies i_0 \in \sigma(u) \text{ for every } u \in\sub_2(I).
    \end{equation}
    Moreover, 
    \begin{equation}\label{wzor:extragraf00}
    G = G_{(i_0)} \; := \; \{ a_{i_0},b_{i_0} \} \cup \starof(i_0).
  \end{equation}
\end{lem}
%


\section{Vergras-like skew}

We start this Section with a presentation of the combinatorial Veronesian $\VerSpace(3,k)$
of \cite{combver}, as, in essence, it will be generalized in the paper.
Besides, this example shows that not every "sensibly roughly presented" perspective 
$\perspace(n,\sigma,{\goth N})$ 
between complete graphs
has necessarily a `Desarguesian axis' nor its skew 
preserves the adjacency of edges of the graphs in question.
\begin{exm}
\label{exm:5}
  \def\a{{\goth a}} \def\b{{\goth b}} \def\c{{\goth c}}
  Let $|X| = 3$, $X = \{ \a,\b,\c \}$.
  Then the combinatorial Veronesian $\VerSpace(X,k) =: {\goth M}$ is 
  a $\binkonf(k,+2)$-configuration; 
  its point set is the set $\msub_k(X)$ of the $k$-element multisets with 
  elements in $X$
  and the lines have form $e X^s$, $e\in \msub_{k-s}(X)$.
  $\VerSpace(X,1)$ is a single line, $\VerSpace(X,2)$ is the Veblen configuration,
  and $\VerSpace(X,3)$ is the known Kantor configuration 
  (comp. \cite[Prop's. 2.2, 2.3]{combver}, \cite[Repr. 2.7]{klik:VC}). 
  Consequently, we assume $k > 3$.
  The following was noted in \cite[Fct. 4.1]{klik:binom}:
  \begin{ctext}
    The $K_{k+1}$ graphs freely contained in $\VerSpace(X,k)$ are the sets
    $X_{\a,\b} := \msub_k(\{ \a,\b \})$, 
    $X_{\b,\c} := \msub_k(\{ \b,\c \})$, and 
    $X_{\c,\a} := \msub_k(\{ \c,\a \})$.
  \end{ctext}
  In particular, $\goth M$ freely contains two complete subgraphs $X_{\a,\b}$, $X_{\c,\a}$, which cross
  each other in $p = \a^k$. We shall present $\goth M$ as a perspective between
  these two graphs.
  Let us re-label the points of $\VerSpace(X,k)$:
  \begin{ctext}
    $c_i = \b^i \a^{k-i}$, $b_i = \c^i \a^{k-i}$, $i\in\{1,\ldots,k\} =: I$,\space
    $e_{i,j} = c_i \oplus c_j$, $\{ i,j \}\in \sub_2(I)$.
  \end{ctext}
  Assume that $i < j$, then $\LineOn(c_i,c_j) = \a^{k-j}\b^{i} X^{j-i}$, so $e_{i,j} = \a^{k-j} \b^i \c^{j-i}$.
  Clearly, $p \oplus c_i = b_i$ so, the map 
    $\big(c_i \mapsto b_i,\; i\in I\big)$
  is a point-perspective.
  Let us define the permutation $\zeta$ of $\sub_2(I)$ by the formula
  \begin{ctext}
  $\zeta(\{ i,j \}) = \{ j-i,j \}$ when $1\leq i < j \leq k$.
  \end{ctext}
  It is seen that $\zeta = \zeta^{-1}$.
  After routine computation we obtain  
  $b_i \oplus b_j = e_{\zeta(\{ i,j\})}$ whenever $i < j$; moreover, 
  in this representation the axial configuration consists of the points
  in $\b \c \msub_{k-2}(X)$ so, it is isomorphic to $\VerSpace(X,k-2)$.
  Consequently, 
\begin{fact}
  $\VerSpace(X,k) \cong \perspace(k,\zeta,{\VerSpace(X,k-2)})$. \myend
\end{fact}
%
\end{exm}
Recall (cf. \cite[Thm. 5.9]{combver}) that for $k > 3$ the structure $\VerSpace(X,k)$
cannot be embedded into any desarguesian projective space.

Let us present a construction of a class of bijections of $\sub_2(X)$ with $|X| = n <\infty$ generalizing the above.
\begin{constr}\label{def:skos:zejergen}
  Let $X$ be a set with $|X| = n < \infty$ and let $\Phi = (\phi_n,\ldots,\phi_{3})$ be a sequence of permutations 
  defined on subsets of $X$ such that
  \begin{multline}\label{def:seqstep}
    \dom(\phi_j) = \rng(\phi_j),\quad \text{ and }\quad 
    |X\setminus\dom(\phi_n)|=1,\; \dom(\phi_{j-1})\subset\dom(\phi_j), \\ |\dom(\phi_j)\setminus\dom(\phi_{j-1})|=1
    \text{ for } j = n,\ldots,4.
  \end{multline}
  We add one item more: $\phi_2 = \id_{\{ 1 \}}$.
  \par
  With such a $\Phi$ we associate a permutation $\sigma_\Phi$ defined as follows.
  Note, first, that $|\dom(\phi_j)| = j-1$. Let us order the elements of $X = (x_n,x_{n-1},\ldots,x_2,x_1)$
  so as $x_j\notin\dom(\phi_j)$. Finally, we set
  \begin{equation}\label{def:seqstep2perm}
    \sigma_\Phi \colon \sub_2(X)\ni \{ x_j,x_i \} \longmapsto \{ x_j,\phi_j(x_i) \} \text{ for }1\leq i < j \leq n.
  \end{equation}
  Then, clearly, $\sigma_\Phi \in \bijfam_{\sub_2(X)}$.
\end{constr}
  Note that, since $|\dom(\phi_2)|=1$,  $\phi_2$ is the identity on $\{ x_1 \}$
  and since $|\dom(\phi_3)|=2$, $\phi_3$ is either the identity on $\{ x_1,x_2 \}$ or it is a transposition.
\par
Clearly, $\Phi^{-1} = (\phi_n^{-1},\ldots,\phi_2^{-1})$ satisfies \eqref{def:seqstep}
and $\sigma_{\Phi^{-1}} = \sigma_\Phi^{-1}$.

Let $\Phi$ satisfy \eqref{def:seqstep}, let $(x_n,\ldots,x_1)$ be an associated ordering of $X$ (note: the order
of elements $x_1,x_2$ of $\dom(\phi_3)$ is not uniquely determined!) and write $X_j = \{ x_1,x_2\ldots,x_{j-1} \}$;
then $\dom(\phi_j) = X_j$.
Let 
$\alpha\colon X \longrightarrow Y$ be a bijection. 
We set $y_j = \alpha(x_j)$ for $j=n,...,1$, so 
$Y_j := \alpha(X_j) = \{ y_1,\ldots,y_{j-1} \}$. Consequently, 
$\phi'_j := \phi_j^{\alpha} \in \bijfam_{Y_j}$ for $j=n,\ldots,2$.
Write $\Phi' = (\phi'_n,\ldots,\phi'_2) =: \Phi^\alpha$; clearly, $\Phi^\alpha$ also satisfies \eqref{def:seqstep}.

Let $i < j$ and $u = \{ x_i,x_j \} \in \sub_2(X)$. Then $\overline{\alpha}(u) = \{ y_i,y_j \}$.
We have $\sigma_{\Phi}(u) = \{ \phi_j(x_i),x_j \}$ and 
$\overline{\alpha}\big( \sigma_{\Phi}(u) \big) = \{ \alpha(\phi_j(x_i)),\alpha(x_j)) \}$. 
Next
$\phi'_j(\overline{\alpha}(u)) = \{ y_j,\phi'_j\big(\alpha(y_i)\big) \}$.
Since $\alpha\circ\phi_j = \phi'\circ\alpha_j$ we conclude with the following
\begin{lem}\label{lem:sprzeg1}
  In the notation introduced above we have  $\sigma_{\Phi'} = (\sigma_\Phi)^{\overline{\alpha}}$.
  To say slightly informally:
  \begin{ctext}
    If $\alpha$ is a bijection of $X$ onto $Y$ then ${(\sigma_\Phi)}^{\overline{\alpha}} = \sigma_{\Phi^\alpha}$.
  \end{ctext}
\end{lem}
Then lemma \ref{lem:sprzeg1} allows us to find conditions sufficient to have skews of the form
$\sigma_\Phi$ conjugate under the action of $\bijfam_X$ on $\sub_2(X)$.

Moreover, to classify skews of the form $\sigma_\Phi$ under the action of $\bijfam_{\sub_2(X)}$ it suffices
to classify the items in the sequence $\Phi$; on the other hand this classification reduces to enumerating
possible parameters of cycle decompositions of corresponding `subpermutations'.
This is justifies by the following Lemma.
\begin{lem}\label{lem:sprzeg2}\def\gora{{\sf s}}
  Let $\Phi^\gora = (\phi_n^\gora,\ldots,\phi_2^\gora)$ be a sequence of permutations 
  (defined on subsets of a set $X^\gora$)
  which satisfies suitably \eqref{def:seqstep} for \/ ${}^\gora \in \{ {}', {}'' \}$.
  Assume that for every $j=2,\ldots,n$ the permutations $\phi'_j$ and $\phi''_j$ have the same
  type of cycle decomposition. Then there is a bijection $\Gamma\colon\sub_2(X')\longrightarrow\sub_2(X'')$ such that 
  $\sigma_{\Phi''} = ( \sigma_{\Phi'} )^\Gamma$.
  \begin{note*}
    Our assumptions do not imply that there is a bijection $\alpha$ such that $\Gamma = \overline{\alpha}$.
  \end{note*}
\end{lem}
\begin{proof}\def\gora{{\sf s}}
  Let $X_j^\gora = \dom(\phi_j^\gora)$ for $j=2,\ldots,n$ and $(x_n^\gora,\ldots,x_1^\gora)$ be an associated ordering
  of $X^\gora$. Set $\gamma_0(x'_j)=x''_j$ for $j=1,2,\ldots,n$.
  Next, for every $j=n,\ldots,2$, since $\phi'_j$ and $\phi''_j$ have the same type, 
  there is a bijection $\gamma_j\colon X'_j\longrightarrow X''_j$ such that $\phi''_j = \big( \phi'_j \big)^{\gamma_j}$
  i.e. $\phi''_j(\gamma_j(x'_i)) = \gamma_j(\phi'_j(x'_i))$ for every $i<j$.
  \par
  Finally, let 
  $1\leq i < j \leq n$ and $u = \{ x'_i,x'_j \}\in \sub_2(X')$ be arbitrary.
  We define 
  $\Gamma(u) := \{ x''_j,\gamma_j(x'_i) \} = \{\gamma_0(x'_j),\gamma_j(x'_i) \}$.
  It needs a routine computation to show that 
  $\sigma_{\Phi''} ( \Gamma(u) ) = \Gamma ( \sigma_{\Phi'}(u) )$, which is our claim.
\end{proof}
%
Consequently, without loss of generality in most parts we shall restrict to the case
\begin{equation}\label{def:seqstep-0}
 X = I_n, \Phi = (\phi_n,\phi_{n-1},\ldots,\phi_2) \text{ where }
 \phi_j\in\bijfam_{I_{j-1}} \text{ for } j=n,n-1,\ldots,2. 
\end{equation}
Two types of such maps are crucial:
\begin{description}
\item[each $\phi_j$ is a symmetry:] $\phi_j(i) = j-i$; we denote the associated skew $\sigma_\Phi$ by 
  $\zeta = \zeta_n$;
\item[each $\phi_j$ is a cycle:] $\phi_j = (1,2,\ldots,j-1)$.
\item[every $\phi_j$ is the identity on $I_{j-1}$:] then $\sigma_\Phi$ is the identity on $\sub_2(I)$.
\end{description}

Clearly, $\zeta = \zeta^{-1}$.
The following counterpart to (a forthcoming) \ref{lem:seq-perm} and a slight restriction to  \ref{lem:sprzeg1}
will be useful
\begin{lem}\label{lem:sprzeg3}
  Let $\Phi$ satisfy \eqref{def:seqstep-0} and $\alpha\in\bijfam_{I_n}$. The following conditions are equivalent
  \begin{enumerate}[(i)]\itemsep-2pt
  \item $\Phi^\alpha$ satisfies \eqref{def:seqstep-0},
  \item $\alpha = \id_{I_n}$ or $\alpha$ is the transposition $(1,2)$.
  \end{enumerate}
\end{lem}
\begin{proof}
  If suffices to observe that we need $\alpha(I_j) = I_{j-1}$ for $j=n,\ldots,3$.
\end{proof}

In the case when $\sigma = \sigma_\Phi$ we have the following specialization of \ref{lem:nextgraf00}.
\begin{prop}\label{prop:nextgraf00}\label{zejer:nextgraf}
  Let $\Phi$ satisfy \eqref{def:seqstep-0} with $X = I_n$
  and ${\goth M} = \perspace(n,\sigma_\Phi,{\goth N})$ for a suitable 
  configuration $\goth N$ defined on $\sub_2(I_n)$.
  The following conditions are equivalent
  \begin{sentences}\itemsep-2pt
  \item
    $\goth M$ freely contains a complete $K_{n+1}$-graph $G\neq K_{A^\ast},K_{B^\ast}$.
  \item
    There is $i_0 \in I$ such that 
    $\starof(i_0)$  
    is a collinearity clique in $\goth N$ freely contained in it,
    $\Phi$ satisfies
    \begin{equation}\label{def:extragraf0zet}
      i_0 < j \implies \phi_j(i_0) = i_0, \text{ for all } j\in I_n, 
    \end{equation}
    and 
    \begin{equation}\label{wzor:extragraf0zet}
      G = G_{(i_0)} \; := \; \{ a_{i_0},b_{i_0} \} \cup \starof(i_0).
    \end{equation}
  \end{sentences}
\begin{note*}
  $G_{(n)}$ is a `new' complete free subgraph in $\goth M$ iff $\starof(n)$
  is a collinearity clique in $\goth N$.
\end{note*}
\begin{note*}
  Assume that 
  $\Fix(\phi_n)=\emptyset$, or $\Fix(\phi_n)\cap\Fix(\phi_{n-1}) =\emptyset$ and $n-1\notin\Fix(\phi_n)$
  (this holds, in particular, when $\Fix(\phi_j)\neq\emptyset$ only for $j=2$). 
  Then $G_{(n)}$ is a `unique new' complete free subgraph of
  $\goth M$ iff $\starof(n)$ is a collinearity clique in $\goth N$.
\end{note*}
\end{prop}
\begin{proof}
  In view of \ref{lem:nextgraf00} it suffices to decide when $\sigma_\Phi$ satisfies the condition \eqref{war:nextgraf00}.
  Let $u = \{j,i_0\}\in\sub_2(I_n)$. If $j< i_0$ then $\sigma_\Phi(u) = \{i_0,\phi_{i_0}(j) \} \ni i_0$, without any
  additional stipulation.
  Let $i_0 < j$; then $\sigma_\Phi(u) = \{ j,\phi_j(i_0) \}$ and thus $i_0 \in \sigma_\Phi(u)$ iff $\phi_j(i_0)= i_0$.
  This completes our proof.
\end{proof}
%

Note that the class of skews of the form $\sigma_\Phi$ determined by sequences of permutations 
is essentially distinct from the skews determined by permutations defined on the segments $I_n$.
\begin{lem}\label{lem:seq-perm}
 Let $n>3$, $\alpha\in\bijfam_{I_n}$, and a sequence $\Phi$ of permutations defined on $I_j$ ($j=n-1,..,2$) 
 satisfy \eqref{def:seqstep-0}. The following conditions are equivalent.
 \begin{sentences}\itemsep-2pt
 \item\label{seq-perm:i}
   The equality $\sigma_\Phi = \overline{\alpha}$ holds. 
 \item\label{seq-perm:ii}
   Either 
   \begin{enumerate}[\rm a)]\itemsep-2pt
   \item
     $\alpha = \id_{I_n}$ and $\phi_j = \id_{I_{j-1}}$ for $j=n,\ldots,2$ (then $\overline{\alpha} = \id_{\sub_2(I_j)}$),
   \item[or]
   \item
     $\alpha = \id_{I_n \setminus I_2} \cup \big( 1,2 \big)$ and $\phi_j = \id_{I_{j-1}}$ for $j=n,n-1,\ldots,4,2$,
     $\phi_3 = (1,2)$.
   \end{enumerate}  
 \end{sentences}
\end{lem}
\begin{proof}
  Assume \eqref{seq-perm:i}. Take arbitrary $j=n,\ldots,4$ and $i< j$.
  Let us set $u_i = \{ n,i \}$; then $\overline{\alpha}(u_i) = \{\alpha(j),\alpha(i)\}$ and 
  $\sigma_\Phi(u_i) = u_{\phi_j(i)}$. There are at least two distinct $i',i''<j$ and then
  $\{ \alpha(j) \} = \overline{\alpha}(u_{i'}) \cap \overline{\alpha}(u_{i''}) = u_{\phi_j(i')} \cap u_{\phi_j(i'')}
  = \{ j \}$. By definition, $\phi_2(1) = 1$. Now, let us pay our attention to $\phi_2$: we need 
  $\{ \alpha(3),\alpha(2)\} = \{ 3,\phi_3(2) \}$ and $\{ \alpha(3),\alpha(1)\} = \{ 3,\phi_3(1) \}$ which is valid
  when $\alpha(3) = 3$ and both when $\phi_3 = (1)(2)$ and when $\phi_3 = (1,2)$, which is noted in \eqref{seq-perm:ii}.
  \par
  Justifying the implication \eqref{seq-perm:ii}$\implies$\eqref{seq-perm:i} is a matter of a simple computation
  similar to the above.
\end{proof}

Configuration 
$\perspace(3,\overline{(1,2)(3)},{\GrasSpace(I_3,2)}) = \perspace(3,\zeta_3,{\GrasSpace(I_3,2)}) = \VerSpace(3,3)$ 
(see Example \ref{exm:5})
is known: it is exactly the Kantor ${10}_3G$-configuration.

As an immediate consequence of \ref{zejer:nextgraf} we obtain
\begin{cor}
    The structure $\perspace(n,\zeta_n,{\GrasSpace(I_n,2)})$
    freely contains exactly three $K_{n+1}$-graphs.
\end{cor}

Let us make the following immediate observation
\begin{lem}\label{lem:zamiana:azb}
  Let $\goth N$ be a $\binkonf(n,0)$-configuration defined on $\sub_2(I_n)$ and let $\Phi$ satisfy \eqref{def:seqstep-0}.
  Clearly, the $\sigma_\Phi$-image
  $\sigma_\Phi\big( {\goth N} \big)$ of $\goth N$ is a $\binkonf(n,0)$-configuration.
  Then the (involutory) map
  \begin{equation}\label{zamiana:azb}
  {\mathscr S}\colon\left\{\begin{array}{ccc}
    a_i & b_{j} & c_u
    \\
    \downarrow & \downarrow & \downarrow
    \\
    b_i & a_j  & c_{\sigma_\Phi(u)}
  \end{array}
  \right.
  \quad \text{ for all } i,j \in I_n, \;\; u\in \sub_2(I_n)
  \end{equation}
  is an isomorphism of $\perspace(n,\sigma_\Phi,{\goth N})$ onto 
  $\perspace(n,\sigma_\Phi^{-1},{\sigma_\Phi\big( {\goth N} \big) })$.
  $\mathscr S$ maps $\starof(n)$ onto $\starof(n)$.
\end{lem}
In essence, in most parts, $\mathscr S$ is the unique automorphims of $\perspace(n,\zeta,{\goth N})$
(when ${\goth N} = \zeta{\goth N}$).
First, note a technical \def\crosswar#1{{\sf Cross}({$#1$})}
\begin{lem}\label{lem:zeta-rigid0}
  Let ${\goth M} = \perspace(n,\sigma_\Phi,{\goth N})$ for a $\binkonf(n,0)$-configuration $\goth N$
  and a sequence $\Phi$ satisfying \eqref{def:seqstep-0}. 
  Next, let $n > 3$ and $k \in I_n$, $k>3$.
  The following conditions are equivalent:
  \begin{sentences}\itemsep-2pt
  \item\label{zetrig:war1}
    the formula
    \begin{ctext}
    \strut\hfill $\forall i\neq k\; \exists j\neq k\;\; \LineOn(a_k,a_i)$ crosses $\LineOn(b_k,b_j)$ \hfill \crosswar{k}
    \end{ctext}
    holds in $\goth M$;
  \item\label{zetrig:war2}
    $n = k$ or $k<n$ and $\phi_j(k) = k$ for all $j > k$.
  \end{sentences}
\end{lem}
\begin{proof}
  It is evident that \eqref{zetrig:war2} implies \eqref{zetrig:war1}: we take: 
  in \crosswar{k} $j=\phi_k(i)$ when $i<k$ and $j = \phi_i(k)$ when $k < i$.
  \par
  Suppose that \crosswar{k} holds for  $3< k <n$. Note that \crosswar{k} means, in fact the following
  \begin{ctext}
    for all $i\neq k$ it holds $\{ k,i \} = \sigma_\Phi(\{  k,j \})$ for some $j\neq k$ i.e. 
    $k \in \sigma_\Phi^{-1}(\{ k,i \})$.
  \end{ctext}
  As in the proof of \ref{zejer:nextgraf} we derive from it $\phi_i^{-1}(k) = k$ for $k>i$, which 
  proves our claim.
\end{proof}
As a corollary to \ref{lem:zeta-rigid0} we obtain the following rigidity property:
\begin{prop}\label{prop:zeta-rigid}
  Let $\goth M$ be as in \ref{lem:zeta-rigid0} with $n > 3$ such that there is no integer $k$ such that $\phi_j(k) = k$
  for all $j > k$.
  Assume that $f\in\Aut({\goth M})$ with $f(p) = p$.
  Then either $f=\id$ or $f = {\mathscr S}$ and ${\goth N} = \sigma_\Phi({\goth N})$.
\end{prop}
\begin{proof}
  Evidently, either $f(A) = A$ or $f(A) = B$ (in the notation of \ref{def:pers}).
  From \ref{lem:zeta-rigid0} we obtain $f(a_n) = a_n$. Then, let us restrict $\goth M$ to points with indices
  in $I_{n-1}$; in this structure \crosswar{n-1} holds and therefore $f(a_{n-1}) = a_{n-1}$ as well.
  Step by step we get $f(a_i) = a_i$ for $3< i \leq n$. 
  Next, we look at $c_{4,2} = a_4 \oplus a_2$, it goes under $f$ onto $a_4 \oplus a_{\alpha(2)} = c_{4,\alpha(2)}$ for 
  a permutation $\alpha\in \bijfam_{I_3}$. Simultaneously, $c_{4,2} = b_4\oplus b_2$ and thus 
  $c_{4,\alpha(2)} = c_{4,4-\alpha(2)}$ which gives $\alpha(2) = 2$.
  Similarly we compute $\alpha(3) = 3$ and $\alpha(1) = 1$.
  \par
  If $f(A) = B$ the reasoning is provided analogously; we obtain $f(a_i) = b_i$ for $3<i\leq n$ and then $f(a_i) = b_i$
  for {\em all} $i\in I_n$.
\end{proof}
In view of \ref{prop:iso0} this yields, in particular, 
\begin{lem}\label{lem:zeta-rigid}
  Let $\varphi\in \bijfam_{I_n}$, $n > 3$. If $\zeta_n^{\varphi} = \zeta_n$ then $\varphi = \id_{I_n}$.
\end{lem}

\bigskip
At the end of this section we shall show how to re-present a skew perspective 
${\goth M} = \perspace(n,\sigma_\Phi,{\goth N})$
as a perspective between its two simplexes other than $K_{A^\ast},K_{B^\ast}$. 
The procedure will be sketched with $\sigma_\Phi = \zeta_n$
and $G_{(n)}$ considered as a third  free $K_{n+1}$-subgraph of $\goth M$; this case is presented as an example.
We hope that this example will explain the general idea of the procedure.
%

%

%
So, let us suppose that $\starof(n)$ is a free clique in $\goth N$. In this case $\goth M$
can be presented as a perspective between two other simplexes contained in $\goth M$:
between $A^\ast\setminus \{a_n\} =: A$ and $G_{(n)}\setminus \{ a_n \} =: D$, with $q = a_n$ as the 
centre of the perspective.
Let us determine the skew of this perspective and its axis.

First, we `renumber' the points in $A = \{ a'_1,\ldots,a'_n \}$; 
next we number the points in $D = \{ d_1,\ldots,d_n \}$ so as
\begin{equation}\label{numery1}
  d_i \in \LineOn(p,a'_i), \quad i \in I_n.
\end{equation}
This is done as follows:
\begin{equation}\label{numery2}
\begin{array}{cccc}
   a_1 & \ldots & a_{n-1} & p
   \\
   \shortparallel & \ldots & \shortparallel & \shortparallel
   \\
   a'_1 & \ldots & a'_{n-1} & a'_n
   \\
   \downarrow & \ldots & \downarrow & \downarrow
   \\
   d_1 & \ldots & d_{n-1} & d_n
   \\
   \shortparallel & \ldots & \shortparallel & \shortparallel
   \\
   c_{1,n} & \ldots & c_{n-1,n} & b_n
\end{array} .
\end{equation}
Then we set 
\begin{ctext}
  $e_{i,j} := a'_i \oplus a'_j$.
\end{ctext}
From the definitions we get
\begin{equation}\label{numery3}
  e_{i,j} = c_{i,j},\; e_{i,n} = b_i \text{ for all } i,j <n,\; i\neq j.
\end{equation}
Finally, we compute for $i,j < n$:

\begin{math}
  d_i \oplus d_j = c_{i,n} \oplus c_{j,n} = c_{\varrho^{-1}_0(\{i,j\})} = e_{\varrho^{-1}_0(\{i,j\})} 
\end{math}
for a map $\varrho_0\colon \sub_2(I_{n-1}) \longrightarrow \sub_2(I_{n-1})$.
The map $\varrho_0$ is entirely determined by the configuration $\goth N$.

To accomplish determining $\varrho$ we must compute
$d_i \oplus d_n$ and compare it with suitable $e_{i',n}$:
\begin{math}
  \text{Recall: } c_{n-i,n} = c_{\zeta(\{ i,n \})} = b_i \oplus b_n.
  \text{ Thus }
  e_{i,n} = b_i = b_{n-(n-i)} = b_n \oplus c_{n-i,n} = d_n \oplus d_{n-i}.
\end{math}

This can be noted as
\begin{math}
  d_i \oplus d_n = e_{n-i,n} = e_{\varrho^{-1}(\{ i,n \})}.
\end{math}
Summarizing, we see that the following defines $\varrho$:
\begin{equation}\label{numery8}
  \varrho^{-1}(\{ i,j \}) =
    \left\{
    \begin{array}{lll}
      \{ i',j' \}& \text{iff } c_{i,n}\oplus c_{j,n} = c_{i',j'}& \text{for } i,j <n
      \\
      \{ n-i,n \} && \text{for }i<n,\; j=n
    \end{array}
    \right. .
\end{equation}

At the very end we characterize the axis $\goth K$ of our perspective: the subconfiguration of $\goth M$
with the points in $E:= \{ e_{i,j}\colon 1\leq i < j \leq n \}$. To do so it suffices to make use the 
following consequence of \eqref{numery3}:
$E = (E \cap C) \cup (B \setminus \{ b_n\}) = 
(C \setminus \starof(n))\cup(B\setminus\{ b_n \})$.
So, $\goth K$ contains all the lines of $\goth N$ which miss $\starof(n)$:
\begin{equation}\label{numery6}
  \text{if } i,j,k,l<n \text{ then } e_{i,j}\oplus e_{k,l} = c_{i,j}\oplus c_{k,l}
  (= c_{s,t} = e_{s,t} \text{ for some } s,t<n).
\end{equation}  
And for $i < j < n$ we have 
\begin{equation}\label{numery7}
  e_{i,n} \oplus e_{j,n} = b_i \oplus b_j = c_{j-i,j} = e_{j-i,j}.
\end{equation}
Conditions \eqref{numery6} and \eqref{numery7} fully characterize the structure $\goth K$, so we obtain

\begin{prop}\label{prop:thesiaos}
  Let ${M}=\perspace(n,\zeta,{\goth N})$ and $\starof(n)$ be a free clique in $\goth N$. Then
$$
  {\goth M} = \perspace(n,\varrho,{\goth K}),
$$
  where $\goth K$ is characterized by \eqref{numery6} and \eqref{numery7}, while $\varrho$ is 
  defined by \eqref{numery8}.
\end{prop}

As a particular instance of the investigations above let us substitute ${\goth N} = \GrasSpace(I_n,2)$;
then $\varrho_0 = \overline{\id_{I_{n-1}}}$.
To make an impression how far ``non-Veblenian'' figures 
may $\goth K$ contain we present in Figure \ref{fig:K} the schema of
a fragment of $\goth K$, when ${\goth N} = \GrasSpace(I_n,2)$.

Besides, with the help of \ref{prop:iso0} we get that there is no automorphism of $\goth M$
which maps $p$ onto $q$. 
Moreover, $\GrasSpace(I_n,2)$ contains $L := \sub_2(\{ 1,2,n \})$ as a line, while 
for $n > 3$ the set $\zeta(L)$ is not any line of $\GrasSpace(I_n,2)$ so, 
$\GrasSpace(I_n,2) \neq \zeta({\GrasSpace(I_n,2)})$.
And therefore, from \ref{lem:zeta-rigid} we conclude with the following
\begin{cor}
  Let $f \in \Aut({\perspace(n,\zeta,{\GrasSpace(I_n,2)})})$ and $n > 3$. Then $f = \id$.
\end{cor}

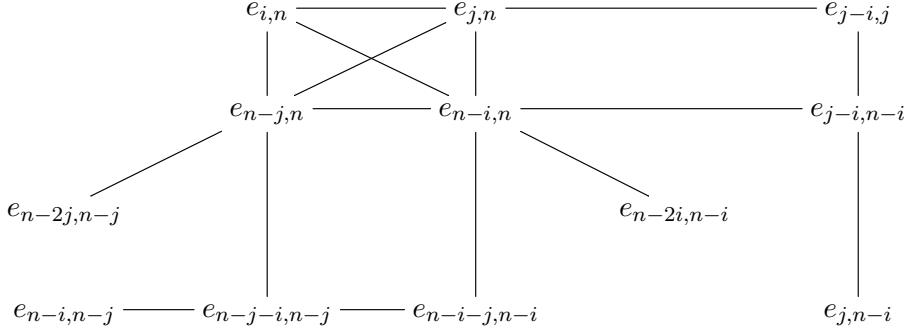
\begin{figure}
\begin{center}
\xymatrix{
{}
&
e_{i,n}
\ar@{-}[r]
\ar@{-}[d]
\ar@{-}[dr]
&
e_{j,n}
\ar@{-}[rr]
\ar@{-}[d]
\ar@{-}[dl]
&
{}
&
e_{j-i,j}
\ar@{-}[d]
\\
{}
&
e_{n-j,n}
\ar@{-}[dd]
\ar@{-}[r]
\ar@{-}[dl]
&
e_{n-i,n}
\ar@{-}[rr]
\ar@{-}[dd]
\ar@{-}[dr]
&
{}
&
e_{j-i,n-i}
\ar@{-}[dd]
\\
e_{n-2j,n-j}
& & &
e_{n-2i,n-i}
\\
e_{n-i,n-j}
\ar@{-}[r]
&
e_{n-j-i,n-j}
\ar@{-}[r]
&
e_{n-i-j,n-i}
&
{}
&
e_{j,n-i}
}
\end{center}
\caption{Let $i<j<n$; then $n-j < n-i$. Moreover, let $i < n-j$ (then $j < n-i$)
and $j < n-j$ (then $i < n-i$). Note that we need $n>4$ to draw such a figure!\myend}
\label{fig:K}
\end{figure}


A more general analysis concerning {\bfseries\em all} the structures $\perspace(n,\sigma_\Phi,{\goth N})$
and a presentation of them as perspectives between other their (admissible) complete subgraphs is skipped over:
it strongly depends on the sequence $\Phi$ and it seems that nothing valuable can be proved in such a general 
setting -- except this, what was already presented.

\section{$n = 4$: the axis is the Veblen Configuration}\label{sec:veblen}

In this section we present a classification of configurations  $\perspace(4,\sigma_\Phi,{\goth N})$   
with $\Phi = (\phi_4,\phi_3,\phi_2)$, $\phi_4\in\bijfam_{I_3}$, $\phi_3\in\bijfam_{I_2}$ and $\phi_2(1) = 1$;
then
$\goth N$ is a $\konftyp((6,2,4,3)$-configuration i.e. $\goth N$ is the Veblen (Pasch) configuration
suitably labelled.
Let us quote after \cite{klik:VC} definitions of the labellings of the Veblen configuration defined on $\sub_2(I_4)$
together with the {\em star-triangles} $\starof(i)$ contained in them:
($Y\in\sub_3(I_4)$: $\topof(Y) := \sub_2(Y)$; $i_0\in I_4$:  $\topof(i_0) := \topof(I_4\setminus\{ i_0 \})$)
\begin{sentences}
\itemindent18.5ex
\def\labelsentence{veblen type \upshape(\thesentence):}
\item\label{vebtyp:1}
  $\GrasSpace(I_4,2)$ -- all four $\starof(i)$ with $i \in I_4$.
\item\label{vebtyp:2}
  $\GrasSpacex(I_4,2)$: its lines are the $\varkappa$-images of the lines of $\GrasSpace(I_4,2)$ 
  and we briefly write $\GrasSpacex(I_4,2) = \varkappa({\GrasSpace(I_4,2)})$ --
  no star-triangle.
\item\label{vebtyp:3}
  $\VeblSpace(2) = \big\langle \sub_2(I_4), \big\{ \topof(3),\; \topof(4), 
    \{ \skros(1,4),\skros(3,4),\skros(2,3) \}, \newline
    \{ \skros(2,4),\skros(3,4),\skros(1,3) \}
  \big\} \big\rangle$ -- $\starof(4)$ and $\starof(3)$ are  its unique star-triangles.
\item\label{vebtyp:4}
  ${\cal V}_5 = \big\langle \sub_2(I_4),\big\{ \topof(4), \{ \skros(1,3),\skros(2,4),\skros(3,4) \},
    \{\skros(1,2), \skros(1,4),\skros(3,4)\}, \newline
    \{ \skros(1,3),\skros(2,4),\skros(3,4)\} \big\} \big\rangle$ --
    $\starof(4)$ is its unique star-triangle;
\item\label{vebtyp:5}
  $\varkappa(\VeblSpace(2)) =: {\cal V}_4$ -- no star-triangle;
\item\label{vebtyp:6}
  $\varkappa({\cal V}_5) =: {\cal V}_6$ -- no star-triangle.
\end{sentences}
For any Veblen configuration $\goth V$ defined on $\sub_2(I_4)$ there is an isomorphism $\overline{\alpha}$
with $\alpha\in\ \bijfam_{I_4}$ of $\goth V$ and a one (exactly one!) of the above six.
To avoid technical troubles we have slightly changed our labelling in comparison 
with the definitions given in \cite{klik:VC}:
we have applied permutation $(1,3)(2,4)$ to the `original' labeling in case \eqref{vebtyp:3}, and
$(1)(2)(3,4)$ in case \eqref{vebtyp:4}.
\par
Evidently, for every $\varphi\in \bijfam_{I_4}$ the $\varphi$-image of any structure $\goth V$ of the above list is 
again a Veblen configuration isomorphic to $\goth V$; but (e.g.) $\perspace(4,\zeta_4,{\goth V})$
and $\perspace(4,\zeta_4,{\varphi({\goth V})})$ may stay non-isomorphic.
One can enumerate all cases when $\varphi({\goth V})$ has $\starof(4)$ as a triangle, but this will not suffice
to obtain a reasonable classification of the structures 
$\perspace(4,\sigma_\Phi,{\varphi({\goth V})})$.

Actually, we need to determine {\em all} labellings of the Veblen configuration
by the elements of $\sub_2(I_4)$.
%
%
By a way of an example on Figure \ref{qver4:gx} 
we show 
drawings presenting the schemas of $\perspace(4,\zeta,{\GrasSpacex(I_4,2)})$ 
and of $\perspace(4,\zeta,{\GrasSpacex(I_4,2)})$;
but remember that these are merely two among eighteen!

\def\vebwpis{{\mathscr V}}

\begin{figure}
\begin{center}
\begin{tabular}{cc}
\ifbqpd  \includegraphics[scale=0.38]{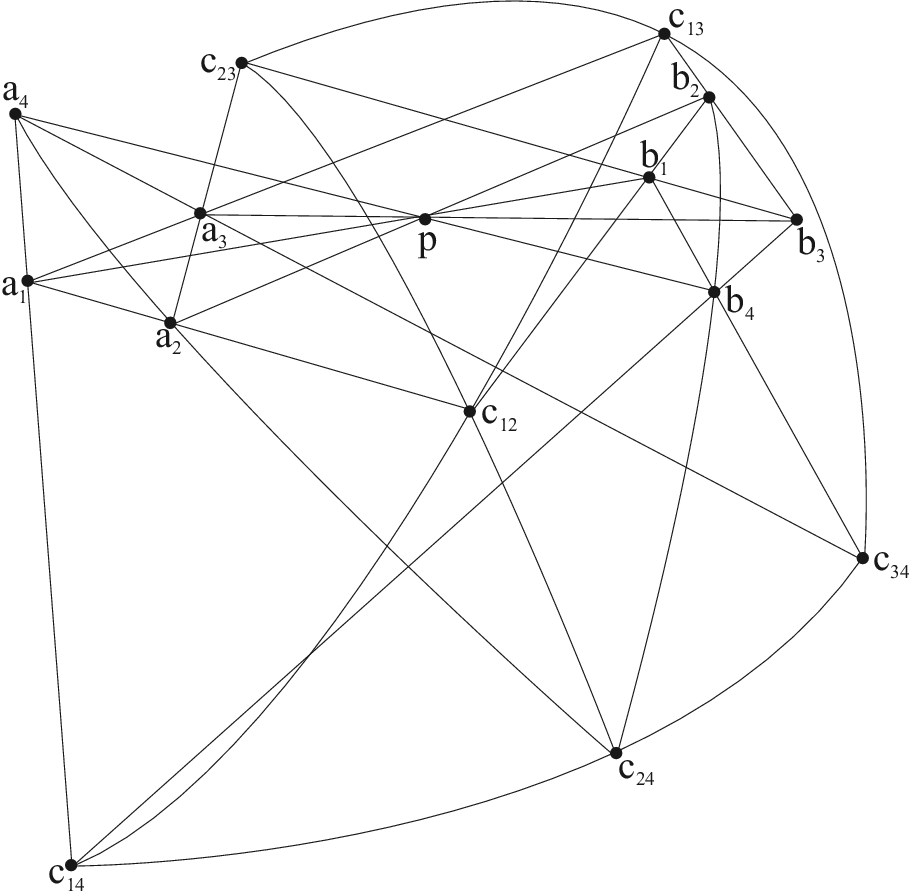}
\else  \includegraphics[scale=0.38]{qver4_gx1.eps}\fi
&
\ifbqpd  \includegraphics[scale=0.38]{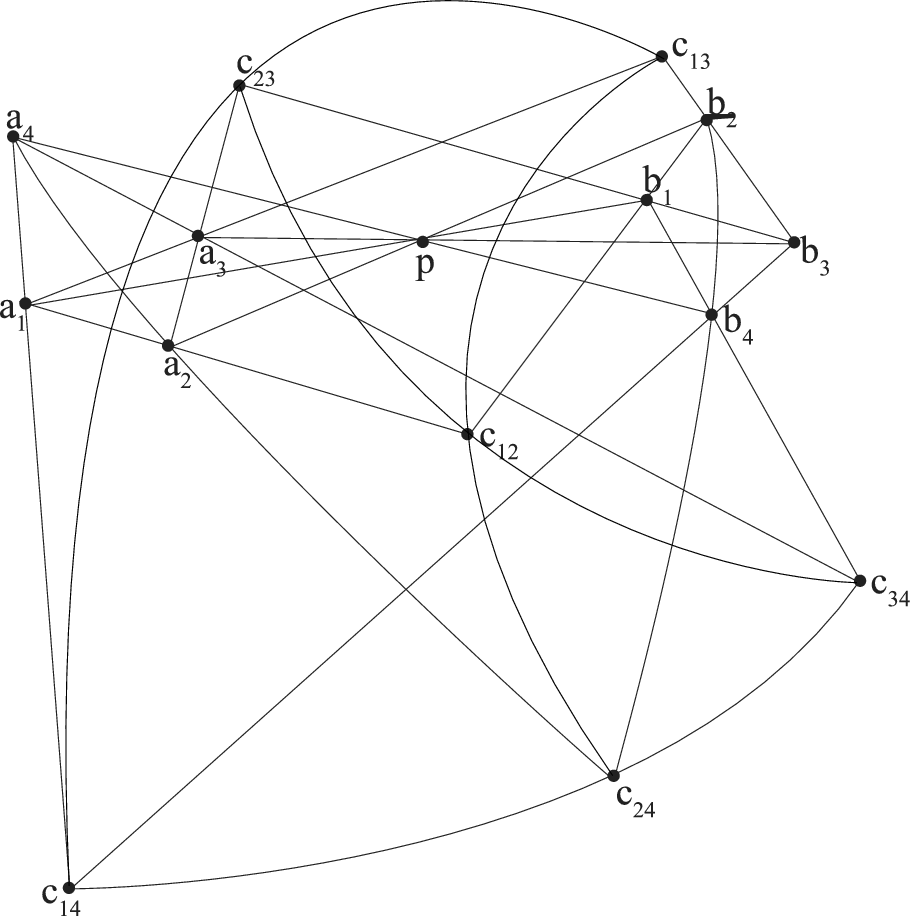}
\else  \includegraphics[scale=0.38]{qver4_gx2.eps}\fi
\end{tabular}
\end{center}
\caption{The structures $\perspace(4,\zeta_4,{\GrasSpacex(I_4,2)}) = \perspace(4,\zeta_4,\vebwpis_6(\id))$ (left)
and $\perspace(4,\zeta_4,\zeta_4({\GrasSpacex(I_4,2)})) = \perspace(4,\zeta_4,\vebwpis_6{\big( (1,2,3)(4) \big)})$ (right) %
(cf. definition of $\vebwpis_s(\mu)$ below).\myend}
\label{qver4:gx}
\end{figure}

Below we introduce a procedure of enumerating respective labellings in a way convenient for our analysis.
Suppose that $\goth V$ contains $\starof(i_0)$ as a triangle, and then it contains $\topof(i_0)$ as a line.
Let us introduce a numbering of the sides of $\starof(i_0)$ and of points of $\topof(i_0)$, invariant under permutations
of $\bijfam_{I_4\setminus \{ i_0 \}}$:
\begin{equation*}
  \LineOn(c_{i,4},c_{j,4}) \sim k \text{ and } c_{i,j} \sim k \text{ iff } \{i,j,k\} = I_4\setminus\{ i_0 \}.
\end{equation*}
Then the definition of $\goth V$ corresponds to a $\mu\in \bijfam_{I_4\setminus\{ i_0 \}}$ with the following rule
\begin{equation}\label{def:mu:trojk-star}
  k \sim \LineOn(c_{i,i_0},c_{j,i_0}) \text{ yields } c_{i,i_0} \oplus c_{j,i_0} = c_{i',j'} \sim \mu(k).
\end{equation}
Next, suppose that $\goth V$ contains $\topof(i_0)$ as a triangle, and then it contains $\starof(i_0)$ as a line.
Analogously to the above we introduce a numbering of the sides of $\topof(i_0)$ and of points of $\starof(i_0)$:
\begin{equation*}
  \LineOn(c_{i,i_0},c_{j,i_0}) \sim k \text{ and } c_{k,i_0} \sim k \text{ iff } \{i,j,k\} = I_4\setminus\{ i_0 \}.
\end{equation*}
Then the definition of $\goth V$ corresponds to a $\mu\in \bijfam_{I_4\setminus\{ i_0 \}}$ with the following rule
\begin{equation}\label{def:mu:trojk-top}
  k \sim \LineOn(c_{i,k},c_{j,k}) \text{ yields } c_{i,k} \oplus c_{j,k} = c_{i',j'} \sim \mu(k).
\end{equation}
Let $\mu\in \bijfam_{I_4\setminus \{i_0\}}$; we write 
$\vebwpis_5(\mu)$ for the Veblen configuration defined by \eqref{def:mu:trojk-star}:
it has $\topof(i_0)$ as a line,
and 
$\vebwpis_6(\mu)$ for the Veblen configuration defined by \eqref{def:mu:trojk-top}:
it has $\starof(i_0)$ as a line.
\par
Next, note that in accordance with the rules above, 
${\goth V} = \vebwpis_s(\mu)$ ($\mu\in \bijfam_{I_4\setminus \{ i_0 \}}$) has another
star-triangle $\starof(i'_0)$ ($s=5$) 
or another top-triangle $\topof(i'_0)$ ($s=6$) iff 
$\mu(i'_0) = i'_0$. In other words, $\mu = (i'_0)(j_1,j_2)$.
It is easy to compute that then ${\goth V} = \vebwpis_s((i_0)(j_1,j_2))$.
\par
Since under every labelling by the elements of $\sub_2(I_4)$ the Veblen configuration  contains either at least
one top-line or at least one star-line, each Veblen configuration
has either the form $\vebwpis_5(\mu)$ or $\vebwpis_6(\mu)$
for some $\mu\in \bijfam_{I_4\setminus \{i_0\}}$ and $i_0\in I_4$.
So, each Veblen configuration $\goth V$ can be uniquely associated with a permutation
$\mu\in \bijfam_4$ with at least one fixed point (not a derangement of $I_4$)
and a `switch' $s \in\{ 5,6 \}$
so as ${\goth V} = \vebwpis_s(\mu)$.
The following is evident
%
%
\begin{ctext}
  $\vebwpis_5(\id) = \GrasSpace(I_4,2)$, \quad
  $\vebwpis_5((3)(4)(1,2)) = \VeblSpace(2)$,\quad
  $\vebwpis_5((4)(1,2,3))  = {\cal V}_5$, 
  \\
  $\vebwpis_6(\id) = \GrasSpacex(I_4,2)$,\quad
  $\vebwpis_6((3)(4)(1,2)) = {\cal V}_4$,\quad
  $\vebwpis_6((4)(1,2,3)) = {\cal V}_6$.
\end{ctext}

Let us note the following tricky observation, justified on  Figure \ref{fig:kapa-od-veba}.
%
\begin{fact}
  For every $s\in \{ 5,6 \}$ and $\mu\in \bijfam_4$ with $\Fix(\mu)\neq\emptyset$
  the following holds
  \begin{ctext}
    $\varkappa(\vebwpis_s(\mu)) = \vebwpis_{11-s}(\mu)$.
  \end{ctext}
\end{fact}
The following is evident.
\begin{fact}\label{fct:zlepka-0}
  Let $\mu,\alpha\in\bijfam_{I_4}$, with $\Fix(\mu)\neq\emptyset$, $s \in \{ 5,6 \}$
  and let $\Phi=(\phi_4,\phi_3)$ determine a suitable skew. Then
  \begin{ctext}
    $\overline{\alpha}\big(\vebwpis_s(\mu)\big) = \vebwpis_s(\mu^\alpha)$.
  \end{ctext}
  Consequently, if $\alpha = \id_{I_4}$ or $\alpha = (1,2)(3)(4)$ (cf. \ref{lem:sprzeg3}) then
  \begin{ctext}
    $\overline{\alpha}\big( \perspace(4,\sigma_\Phi,\vebwpis_s(\mu)) \big) =
    \perspace(4,\sigma_{\Phi^\alpha},\vebwpis_s(\mu^\alpha))$.
  \end{ctext}
  \begin{note*}
    So, enumerating all the $\sigma_\Phi$-perspectives along admissible sequences $\Phi$ and after that along 
    {\em all the admissible} $\mu$ one can cancel one item in every pair $(\Phi,{\Phi}^{(1,2)(3)(4)})$.
  \end{note*}
  \begin{note*}
    If, additionally, $\Phi = {\Phi}^{(1,2)(3)(4)}$ then in the list of considered $\mu$ one can cancel one item in
    every pair $(\mu,{\mu}^{(1,2)(3)(4)})$.
  \end{note*}
  \begin{note*}
    Analogously, in view of \ref{lem:zamiana:azb} we can cancel one item in every pair $(\Phi,\Phi^{-1})$. 
  \end{note*}
\end{fact}

\begin{figure}\def\zakos{{/\mkern-8mu/}}
\begin{footnotesize}
\xymatrix{%
{\{\mu(j),i_0\}\zakos\not\ni\mu(j)}\ar@{-}[r]\ar@{-}[dr]       &  
{\{\mu(k),i_0\}\zakos\not\ni\mu(k)}\ar@{-}[rrrr]\ar@{-}[d] & {\strut} & {\strut} &  {\strut} & 
{\{\mu(i),i_0\}\zakos\not\ni\mu(i)}
\\
{\strut} &  {\{j,k\}\zakos\{i_0,i\}}\ar@{-}[dd]^{k}\ar@{-}[dr]^{j}
\\
{\strut} & {\strut} & {\{i,j\}\zakos\{i_0,k\}} \ar@{-}[uurrr] 
\\
{\strut} & {\{i,k\}\zakos\{i_0,j\}} \ar@{-}[ur]^{i}
} 
\end{footnotesize}
\caption{%
Comparing $\vebwpis_6(\mu)$ and its $\varkappa$-image, $\mu(i_0) = i_0$.
$I_4 = \{ i,j,k,i_0 \}$.
Points on the diagram are denoted following the convention: 
$\text{ value-of-}u \zakos\text{value-property-of-}\varkappa(u)$ with $u\in\sub_2(I_4)$,
where the `starting' structure $\vebwpis_6(\mu)$
has the line $\starof(i_0)$ and the triangle $\topof(i_0)$.
\myend}
\label{fig:kapa-od-veba}
\end{figure}

If ${\goth M}= \perspace(4,\sigma_\Phi,{\goth V})$ freely contains three $K_5$ 
then, in accordance with \cite{STP3K5}, $\goth M$ can be presented as a so called 
{\em system of triangle perspectives} (\STP, in short).
Let us start with a slight reminder of this representation technique of \cite{STP3K5}.
We arrange the vertices of three triangles of $\goth M$: 
$\Delta_1 = \{ a_1,a_2,a_3 \}$, $\Delta_2 = \{ b_1,b_2,b_3 \}$, and $\Delta_3 = \starof(4)$ in three rows of 
a $3\times 3$-matrix in such a way that when we join in pairs points in the same two columns, 
the obtained lines of $\goth M$ have a common point. So obtained three common points form the line $\topof(4)$. 
After that we join points in distinct rows when there is a line in $\goth M$ which joins them:
these lines for every pair of rows should meet in a common point.
On Figure \ref{fig:schemat} we visualize a schema of this procedure.

\begin{figure}
\begin{center}
\begin{minipage}[m]{0.45\textwidth}
    \xymatrix{%
    {\Delta_1:}
    &
    {a_{1}}\ar@{-}[dr]
    &
    {a_{2}}\ar@{-}[dl]
    &
    {a_{3}}\ar@{-}[d]
    \\
    {\Delta_2:}
    &
    {b_{2}}
    &
    {b_{1}}
    &
    {b_{3}}
    \\
    {\Delta_3:}
    &
    {c_{i,4}}
    &
    {c_{j,4}}
    &
    {c_{k,4}}
    }
  \end{minipage}
  \quad\quad\quad
  \begin{minipage}[m]{0.45\textwidth}
    $a_4 \in \LineOn(a_i,c_{4,i}),\; i = 1,2,3$.
    \par
    $b_4 \in \LineOn(b_2,c_{2,4}),\; \LineOn(b_1,c_{3,4}),\;\LineOn(b_3,c_{1,4})$.
    \par
    \strut\par
    {$c_{1,2}\in \LineOn(a_1,a_2),\; \LineOn(b_2,b_1),\; \LineOn(c_{i,4},c_{j_4})$},\newline
    {$c_{2,3}\in \LineOn(a_2,a_3),\; \LineOn(b_1,b_3),\; \LineOn(c_{j,4},c_{k_4})$},\newline
    {$c_{1,3}\in \LineOn(a_1,a_3),\; \LineOn(b_2,b_3),\; \LineOn(c_{i,4},c_{k_4})$}.
  \end{minipage}
\end{center}
\caption{A schema of the diagram of the line $\topof(4)$ in $\perspace(4,\zeta,{\goth V})$, 
$\{i,j,k\} = \{ 1,2,3\}$. $\goth V$ is a labelling of the Veblen configuration which contains
a free triangle $\starof(4)$ and, consequently, $\topof(4)$ as a line.%
\myend}
\label{fig:schemat}
\end{figure}

It is known that after such a representation the obtained structures are (with a few exceptions) isomorphic
when the associated diagrams are isomorphic (can be mapped one onto the other by a permutation of rows and
columns). From Figure \ref{fig:schemat} we read that the diagram is determined by 
the permutation $\{i,j\}\mapsto\{i',j'\}$: $c_{i',j'} = c_{i,4}\oplus c_{j,4}$ with $1\leq i, j \leq 3$.

This technique allows us to formulate a complete characterization of the structures of the form 
$\perspace(4,\sigma_\Phi,{\goth V})$.
%
With a pouring computer-aided computations we obtain the following theorem.
\begin{thm}\label{thm:lista4}{}
  \begin{sentences}\itemsep-2pt
  \item\label{lista4-1}
    There are $104$ pair wise non isomorphic $\konftyp(15,4,20,3)$ configurations of the form 
    $\perspace(4,\sigma_{\Phi},{\goth V})$ which freely contain exactly two $K_5$-graphs where
    $\Phi = (\phi_4,\phi_3)$ satisfies \eqref{def:seqstep-0}, 
    $\sigma_\Phi\neq\overline{\alpha}$ for every $\alpha\in\bijfam_{I_4}$,  and $\goth V$ is a Veblen configuration 
    defined on $\sub_2(I_4)$. 
  \item\label{lista4-2}
    Let $\goth M$ be a skew perspective $\perspace(4,\sigma_\Phi,{\goth V})$ 
    with $\Phi$, $\goth V$ as in \eqref{lista4-1}. 
    Then $\goth M$  freely
    contains at least three $K_5$-graphs iff it is isomorphic to a one of 11 {\STP}'s of \cite{STP3K5}:
    those enumerated in \cite[Classification 2.8]{STP3K5} with (i), (ii), (iv), (v), and (xiii) excluded and
    (ii), (iii) in \cite[Remark 2.10]{STP3K5}.
  \end{sentences}
\end{thm}
\begin{note}\normalfont
  Let us remind, after \ref{lem:seq-perm} and \ref{fct:zlepka-0}, that if $\Phi$ satisfies \eqref{def:seqstep-0}
  and $\sigma_{\Phi} = \overline{\alpha}$ with $\alpha\in\bijfam_{I_4}$ then we can assume that $\alpha = \id$,
  $\sigma_{\overline{\alpha}} = \id_{\sub_2(I_4)}$,
  and all the perspectives of the form $\perspace(4,\id_{\sub_2(I_4)},{\goth V})$ were already classified in
  \cite{maszko:male}.
\end{note}
Practically, explicit enumerating all the $104$ configurations of \eqref{lista4-1} together with values of
parameters which yield isomorphic items does not give any essential information.
However, to complete the course let us briefly say a few words on the subject.
\def\PHI{\mathit{PHI}} \def\MU{\mathit{MU}}
\begin{note}\label{szczegoly1}\normalfont
  Consider 
  ${\goth M}$ =  $\perspace(4,{\sigma_{\Phi}},{\vebwpis_s(\mu)})$ 
  where 
  $\Phi = (\phi_4,\phi_3)$ satisfies \eqref{def:seqstep-0}, $s \in \{ 5,6 \}$, and $\mu\in\bijfam_{I_4}$
  with $\Fix(\mu)\neq\emptyset$.
  Taking into account \ref{fct:zlepka-0} we can restrict ourselves to the following list $\PHI$ of permutations:
  \begin{multline}
   \PHI =  ([(1)(2)(3), (1)(2)], \;[(1)(2)(3),(1,2)], \;[(1) (2,3), \,(1)(2)], \;[(2) (1,3), \,(1,2)], \, 
   \\
     [(1,2)(3), \,(1)(2)], \;[(1,2)(3)], \,(1,2)], 
    \;[(1,2,3), \,(1)(2)], \;[(1,2,3)), \,(1,2)]).
  \end{multline}
  %
  Note that  $\sigma_{\PHI[1]} = \overline{\id_{I_4}}$, 
  $\sigma_{\PHI[4]} = \zeta_4$,
  $(\PHI[i])^{-1} = \PHI[i]$ for $i=2,..,6$, 
  $\sigma_{\PHI[7]}$ has order $3$,
  and $\sigma_{\PHI[8]}$ has order 6.
  Next, the admissible $\mu$ can be arranged into the following sequence $\MU$ of lenght 15:
  \begin{multline}
   \MU = (\,(1)(2)(3)(4), \;  (1,2,3)(4), \; (1,3,2)(4), \; 
   (1,2,4)(3), \; (1,4,2)(3), \; 
   (1,3,4)(2),  \\ (1,4,3)(2), \; 
   (1)(2,3,4), \; (1)(2,4,3), \;
   (1)(2)(3,4), \; (1),(2,4)(3), \; (1)(4)(2,3), \\ (1,4)(2)(3), \;  (1,3)(2)(4), \; (1,2)(3)(4) \,).
  \end{multline} \def\MM{{\mathscr M}}
  Let us denote $\MM(f,s,i) = \perspace(4,\sigma_{\PHI[f]},{\vebwpis_s({\MU[i]})})$ 
  for $f=1,...,8$, $s=5,6$, and $i=1,...,15$.
  In view of \ref{lem:seq-perm} and \ref{fct:zlepka-0} we can assume that $f\neq 1$.

  Suppose that $\MM(f,s,i)$ freely contains at least three $K_5$; then $s=5$. 
  Using \ref{zejer:nextgraf} we get the following.
  If $G_{(4)}$ is an additional $K_5$ then $i\in\{1,2,3,12,14,15\}$, and $f$ is arbitrary.
  If $G_{(3)}$ is a $K_5$ then $f\in\{ 1,2,5,6 \}$ and $i\in \{ 1,4,5,11,13,15 \}$.
  $G_{(2)}$ cannot be a free $K_5$ for any $f$.
  $G_{(1)})$ is a free $K_5$ for $f = 3$ and $i \in \{ 1,8,9,10,11,12 \}$.
  Summing (and using a Maple program which sayes for which parameters the corresponding perspectives
  are pair wise isomorphic) we conclude that $\MM(f,s,i)$ freely contains at least three $K_5$ iff $s=5$ and 
  $(f,i)$ are among the following:
  \begin{multline}\label{listaaa}
   (2,1), (2,2), (2,3),  (2,4), (2,5), (3,1), (3,2), (3,3), (3,8), (3,9), (3,15), (4,1), (4,2),\\
   (4,12), (4,14), (5,1), (5,2), (5,3), (5,4), (6,1), (6,2), (6,4) (6,11), (6,12), (6,15), \\
   (7,1), (7,2), (7,12), (7,14), (7,15),  (8,1), (8,2), (8,3), (8,12), (8,14), (8,15). 
  \end{multline}
  %
  Diagrams of the respective configurations computed by a Maple program prove \eqref{lista4-2} of \ref{thm:lista4}.
  \par
  In the remaining cases, taking into account \ref{lem:zamiana:azb} and \ref{prop:iso0} and determining
  (with the help of a Maple program) when $\sigma_{\PHI[f]}( \vebwpis_{s_2}(\MU[i_2]) ) = \vebwpis_{s_1}(\MU[i_1])$
  we get the following list of parameters of pair wise non isomorphic configurations $\MM(f,s,i)$
  with two free $K_5$:
  \begin{description}\itemsep-2pt
  \item[$f = 2$:] $s=6$ and $i= 1,2,3,4,5,6,7,8,9,10$;  
  \item[$f = 3$:] $s=6$ and $i= 1,2,3,4,5,6,7,8,9,13$;  
  \item[$f = 4$:] $s=5$ and $i= 4,5,6,8,9,10,11,13$,\/ $s=6$ and $i= 1,3,5,9,12,14$; 
  \item[$f = 5$:] $s=6$ and $i= 1,2,3,4,5,6,7,8,9,10$; 
  \item[$f = 6$:] $s=5$ and $i= 6,7,10$, $s=6$ and $i= 1,2,4,6,7,10,11,12,15$, 
  \end{description}
  Note that for $f=7,8$ we cannot get $\MM(f,s,i_i) \cong \MM(f,s,i_2)$ for $i_1\neq i_2$ when the respective
  perspectives have exactly two free $K_5$, so in this case pair wise non isomorphic are all $\MM(f,6,i)$ and
  $\MM(f,5,i)$ with $(f,i)$ not in the list \eqref{listaaa}.
\end{note}
\begin{note}\label{szczegoly2}\normalfont
  Note that if a system $\goth M$ is of type (ii) or (vi) of \cite[Classification 2.8]{STP3K5} then 
  ${\goth M} = \perspace(4,\sigma,{\goth V})$, where $\sigma = \sigma_\Phi$ and $\Phi$ is 
  defined in a way very similar to \eqref{def:seqstep-0}.
  With 
  \begin{ctext} 
  $\phi_1 = \begin{tabular}{c|c|c} 2 & 3 & 4 \\ \hline 3 & 1 & 4 \end{tabular}$ and
  $\phi_2 = \begin{tabular}{c|c} 3 & 4 \\ \hline 1 & 4 \end{tabular}$
  \end{ctext}
  we have $\sigma(\{ 1,i \}) = \{ 2,\phi_1(i) \}$ and $\sigma(\{ 2,j \}) = \{ 3,\sigma_2(j) \}$;
  clearly, $\sigma(\{ 3,4 \}) = \{  3,4\}$.
  For $\goth M$ of type (i), (iv), and (xiii) there is no such a pair $\Phi$.\myend
\end{note}

The automorphisms of the structures $\perspace(4,\sigma_\Phi,{\goth V})$ which freely contain three $K_5$
are determined in \cite{STP3K5} so, there is no need to write them down explicitly here.
Then, again using a Maple program to determine when (in the notation of \ref{szczegoly1}) the equation
$\sigma_{\PHI[f]}(\vebwpis_s(\MU[i])) = \vebwpis_s(\MU[i])$ holds we obtain the following theorem.
\begin{thm}
  Let ${\goth V} = \vebwpis_s(\mu)$ be a Veblen configuration defined on $\sub_2(I_4)$ and
  ${\goth M} = \perspace(4,\sigma_\Phi,{\goth V})$, $s\in\{ 5,6 \}$, $\mu \in \bijfam_{I_4}$,
  and $\Phi = (\phi_4,\phi_3)$ satisfy \eqref{def:seqstep-0}.
  Assume that $\goth M$ contains exactly two $K_5$.
  Then $\Aut({\goth M})$ is nontrivial only when
  \begin{description}\itemsep-2pt
  \item[$\Phi = (\,(2) (1,3), \;(1,2)\,)$ {\normalfont{and}}]  
  %
    \begin{description}\itemsep-2pt
    \item[$s = 5$:] $\mu \in \{ (4)(1,2,3), \, (1)(2,3,4), \, (1)(4)(2,3) \}$,
    \item[$s = 6$:] $\mu \in \{ (4)(1,3,2), \, (1)(2,4,3), \, (1)(4)(2,3)$;
    \end{description}
  \item[$\Phi = (\,(1,2)(3), \;(1,2)\,)$ {\normalfont{and}} ]
    \begin{description}\itemsep-2pt
    \item[$s = 5$ ] $\mu = (1)(2)(3,4)$.
    \item[$s = 6$], $\mu \in \{ (1)(2)(3)(4),\, (1)(2)(3,4), \, (1,2)(3)(4) \}$.
    \end{description}
  \end{description}
  %
  If $\Aut({\goth M})$ is nontrivial then $\Aut({\goth M}) = \{ \id, {\mathscr S}\} \cong C_2$.
\end{thm}
\begin{proof}
  The claim is an immediate consequence of  \ref{prop:zeta-rigid}. 
\end{proof}

\begin{figure}
\begin{center}
\begin{minipage}[m]{0.6\textwidth}
    \xymatrix{%
    {\Delta_1:}
    &
    {a_{3}}\ar@{-}[d]\ar@(l,l)@{-}[dd]
    &
    {a_{2}}\ar@{-}[dr]\ar@{-}[ddr]
    &
    {a_{1}}\ar@{-}[dl]\ar@{-}[ddl]
    \\
    {\Delta_2:}
    &
    {b_{3}}\ar@{-}[dr]
    &
    {b_{1}}\ar@{-}[dl]
    &
    {b_{2}}\ar@{-}[d]
    \\
    {\Delta_3:}
    &
    {c_{3,4}}
    &
    {c_{1,4}}
    &
    {c_{2,4}}
    }
  \end{minipage}
  \end{center}
  \caption{The diagram of the line $\{ c_{1,2}, c_{2,3}, c_{1,3} \} = \topof(4)$
  in $\perspace(4,\zeta,{\VeblSpace(2)}) = \perspace(4,\zeta_4,\vebwpis_5\big( (3)(1,2) \big))$: 
  comp. rules on Figure \ref{fig:schemat}.
\myend} 
\label{fig:PB2}
\end{figure}


\let\bibtem\bibitem

\noindent Author's address:\newline
{Agata Bazylewska-Zejer,
Ma{\l}gorzata Pra{\.z}mowska, Krzysztof Pra{\.z}mowski\\
Institute of Mathematics, University of Bia{\l}ystok\\
ul. Cio{\l}kowskiego 1M\\
15-245 Bia{\l}ystok, Poland\\
e-mail:  {\ttfamily agatazejer@gmail.com},  {\ttfamily malgpraz@math.uwb.edu.pl}, \newline\strut\hfill 
{\ttfamily krzypraz@math.uwb.edu.pl}

\end{document}